\documentclass[11pt]{amsart}
\usepackage{amsfonts}
\usepackage{mathrsfs}
\usepackage{amsmath,amssymb,amscd,amsthm,amsxtra}
\usepackage{latexsym}

\allowdisplaybreaks

\usepackage{graphicx,epstopdf,verbatim}

\newtheorem{theorem}{Theorem}[section]
\newtheorem{lemma}{Lemma}[section]

\newtheorem{remark}{Remark}[section]

\newtheorem{definition}{Definition}[section]
\newtheorem{corollary}{Corollary}[section]

\numberwithin{equation}{section}

\begin{document}

\title
{Commutator Theorems for  Fractional Integral Operators on Weighted Morrey Spaces}

\author{Zengyan Si }
\address{Zengyan Si\\
School of Mathematics and Information Science\\
Henan Polytechnic University\\
Jiaozuo 454000 \\
P. R. China
}

\email{sizengyan@yahoo.cn}

\begin{abstract}
Let $L$ be the infinitesimal generator of an analytic semigroup on $L^2(R^n)$ with Gaussican kernel bounds, and let $L^{-\alpha/2}$ be the fractional integrals of $L$ for $0<\alpha<n.$  For any locally integrable function  $b$,  The commutators associated with  $L^{-\alpha/2}$ are defined by $[b,L^{-\alpha/2}](f)(x)=b(x)L^{-\alpha/2}(f)(x)-L^{-\alpha/2}(bf)(x)$.
When $b\in BMO(\omega)$(weighted $BMO$ space) or  $b\in BMO$, the author obtain the necessary and sufficient conditions for the boundedness of  $[b,L^{-\alpha/2}]$ on weighted Morrey spaces respectively.
 \end{abstract}

%
%



 \maketitle

\section{INTRODUCTION AND MAIN RESULTS}
Morrey \cite{morrey} introduced the classical Morrey spaces to investigate the local behavior of solutions to second order elliptic partial differential equations. Chiarenza and Frasca \cite{chiarenzafrasca} established the boundedness of the Hardy-Littlewood maximal operator, the fractional operator and a  singular integral operator on the Morrey spaces. On the other hand, Coifman and Fefferman \cite{coifman}, Muckenhoupt \cite{muckenhoupt} studied the boundedness of these operator on weighted $L^p$ spaces. Motivated by these work, Komori and Shirai \cite{ks} introduced the following weighted Morrey space and investigated the boundedness of classical
operators in harmonic analysis, that is, the Hardy-Littlewood maximal operator, a Calder\'{o}n-Zygmund operator,
the fractional integral operator, etc.
\par
 Let $1\leq p<\infty$ and $0\leq k<1$. Then for two weights $\mu$ and $\nu$, the weighted Morrey space is defined by
\begin{equation*}
    L^{p,k}(\mu,\nu)=\{f\in L_{loc}^p(\mu): ||f||_{L^{p,k}(\mu,\nu)}<\infty \},
\end{equation*}
where
\begin{equation*}\label{def}
   ||f||_{L^{p,k}(\mu,\nu)}=\sup_Q\left( \frac{1}{\nu(Q)^k}\int_Q |f(x)|^p\mu(x)dx\right)^{\frac{1}{p}}.
\end{equation*}
and the supremum is taken over all cubes $Q$ in $R^n$.
\par If $\mu=\nu,$ then we have the classical Morrey space $L^{p,k}(\mu)$ with measure $\mu$.
When $k=0,$ then $L^{p,k}(\mu,\nu)=L^{p}(\mu)$ is the Lebesgue space with measure $\mu$.
\par
Suppose that $L$ is a linear operator on $L^2 (R^n)$ which generates an analytic semigroup
$e^{-tL}$ with a kernel $p_t(x, y)$ satisfying a Gaussian upper bound, that is,
\begin{equation}\label{1.1}
   |p_t(x,y)|\leq \frac{C}{t^{\frac{n}{2}}}e^{-c\frac{|x-y|^2}{t}}
\end{equation}
for $x, y\in R^n$ and all $t>0.$
\par
For $0<\alpha<n, $ the fractional integral $L^{-\alpha/2}$ of the operator $L$ is defined by
$$L^{-\alpha/2}f(x)=\frac{1}{\Gamma(\frac{\alpha}{2})}\int_0^\infty e^{-tL}(f)\frac{dt}{t^{-\alpha/2+1}}(x).$$
Note that if $L = -\Delta$ is the Laplacian on $R^n$,
then $L^{-\alpha/2}$ is the classical fractional
integral $I_\alpha$ which plays important roles in many fields. It is well known that
$I_\alpha  $ is bounded from $L^p(R^n)$ to $L^q(R^n)$
 for all $p>1, 1/q=1/p-\alpha/n>0$ and is also of weak type $(1, n/(n-\alpha)).$
\par
Let $1\leq p<\infty$ and $\omega$ be a weight function.  A locally integral function $b$ is said to be in $BMO_p(\omega)$ if
 \begin{equation*}
  ||b||_{BMO_p(\omega)}=\sup_Q\left( \frac{1}{\omega(Q)}\int_Q |b(x)-b_Q|^p\omega(x)^{1-p}dx\right)^{\frac{1}{p}}\leq C< \infty,
\end{equation*}
 where $b_Q=\frac{1}{|Q|}\int_Q b(y)dy$ and the supremum is taken over all cube $Q\in R^n.$
 \par
 Let  $\omega \in A_1$, Garc\'{i}a-Cuerva \cite{garcia}  proved that the spaces $BMO_p(\omega)$ coincide, and the norm of $||\cdot||_{BMO_p(\omega)}$ are equivalent with respect to different values of provided that $1\leq p<\infty.$
\par
Let $b$ be a locally integrable function on $R^n$, we consider the commutator $[b,L^{-\alpha/2}]$ defined by
$$[b,L^{-\alpha/2}](f)(x)=b(x)L^{-\alpha/2}(f)(x)-L^{-\alpha/2}(bf)(x).$$
\par  Chanillo \cite{chanillo}  proved that  the commutator $[b, I_\alpha]$ of the multiplication operator by $b\in BMO$ is bounded on $L^p$ for $1<p<\infty.$
\par Duong and Yan \cite{duongyan} proved $[b,L^{-\alpha/2}]$ is bounded from $L^p$ to $L^q, $ where $b\in BMO,1<p<n/\alpha, 1/q=1/p-\alpha/n,0<\alpha <n.$
\par Mo and Lu \cite{molu} proved the multilinear commutator generated by $\vec{b}$ and $L^{-\alpha/2}$ is bounded from $L^p$ to $L^q, $ where $1<p<n/\alpha, 1/q=1/p-\alpha/n,0<\alpha <1$,  $\vec{b}=(b_1,\cdots,b_m), b_i\in BMO,$ for $i=1,\cdots,m.$

\par Lu, Ding and Yan \cite{ludingyan} proved $[b,I_\alpha]$ is bounded from $L^p$ to $L^q$ if and only if $b\in BMO.$

\par
  Wang \cite{wang} proved that $[b,I_\alpha]$ is bounded from $L^{p,k}(\omega)$ to $L^{q,kq/p}(\omega^{1-(1-\alpha/n)q},\omega),$ where
  $b\in BMO(\omega)$, $0<\alpha<n, 1<p<n/\alpha, 1/q=1/p-\alpha/n, 0<k<p/q$ and $\omega^{q/p}\in A_1.$

  Inspired by the above results, we study the boundedness properties of the commutator $[b,L^{-\alpha/2}]$ on weighted Morrey spaces  in this work.
The main theorems are stated as follows.
\begin{theorem}\label{thm1.1}
Let $0<\alpha<n, 1<p<n/\alpha, 1/q=1/p-\alpha/n,$
$0\leq k<p/q,$ $\omega^{\frac{q}{p}}\in A_1$ and $ r_\omega> \frac{1-k}{p/q-k},$  where $r_\omega$ denotes the critical index of $\omega$ for the reverse H\"{o}lder condition. Then the following conditions are equivalent:
\par $(a)$  $b\in BMO(\omega).$
\par
 \par $(b)$ $[b,L^{-\alpha/2}]$ is bounded from $L^{p,k}(\omega)$ to $L^{q,kq/p}(\omega^{1-(1-\alpha/n)q},\omega)$.
\end{theorem}
Specially, when $k=0$ in Theorem \ref{thm1.1}, we get
\begin{corollary}
Let $0<\alpha<n, 1<p<n/\alpha, 1/q=1/p-\alpha/n,$
 $\omega^{\frac{q}{p}}\in A_1$ and $ r_\omega> \frac{q}{p},$  where $r_\omega$ denotes the critical index of $\omega$ for the reverse H\"{o}lder condition. Then the following conditions are equivalent:
\par $(a)$  $b\in BMO(\omega). $
 \par $(b)$  $[b,L^{-\alpha/2}]$ is bounded from $L^{p}(\omega)$ to $L^{q}(\omega^{1-(1-\alpha/n)q})$.
\par Further more, if $L = -\Delta$ is the Laplacian, Then the following conditions are equivalent:
\par $(a^{'})$  $b\in BMO(\omega). $
 \par $(b^{'})$  $[b,I_\alpha]$ is bounded from $L^{p}(\omega)$ to $L^{q}(\omega^{1-(1-\alpha/n)q})$.
\end{corollary}
\begin{theorem}\label{thm1.2}
Let $0<\alpha<n, 0\leq k<p/q, 1/q=1/p-\alpha/n,$   and $1<r,s<\infty$ such that $1<rs<p<n/\alpha$,  $ \omega^{rs}\in A_{p/rs,q/rs}.$ Then the following conditions are equivalent:
\par $(a)$ $b\in BMO. $
 \par$(b)$ $[b,L^{-\alpha/2}]$ is bounded from $L^{p,k}(\omega^p,\omega^q)$ to $L^{q,kq/p}(\omega^q)$.
\end{theorem}
Specially, when $k=0$ in Theorem \ref{thm1.2}, we obtain
%
\begin{corollary}
Let $0<\alpha<n, 1/q=1/p-\alpha/n,$   and $1<r,s<\infty$ such that $1<rs<p<n/\alpha$,  $ \omega^{rs}\in A_{p/rs,q/rs}.$ Then the following conditions are equivalent:
\par $(a)$ $b\in BMO. $
 \par$(b)$ $[b,L^{-\alpha/2}]$ is bounded from $L^{p}(\omega^p)$ to $L^{q}(\omega^q)$.
\par Further more, if $L = -\Delta$ is the Laplacian, Then the following conditions are equivalent:
\par $(a^{'})$   $b\in BMO. $
 \par $(b^{'})$  $[b,I_\alpha]$ is bounded from $L^{p}(\omega^p)$ to $L^{q}(\omega^q)$.
\end{corollary}
\remark{It is easy to see that our results extend the results in \cite{chanillo},\cite{duongyan},\cite{ludingyan},\cite{wang} significantly.}

%
\section{PREREQUISITE MATERIAL}
  Let us first recall some definitions.
\begin{definition}
The Hardy-Littlewood maximal operator $M$ is defined by
 $$M(f)(x)=\sup_{x\in Q}\frac{1}{|Q|}\int_Q |f(y)|dy.$$
   Let $\omega$ be a weight. The weighted maximal operator $M_\omega$ is defined by
  $$M_\omega(f)(x)=\sup_{x\in Q}\frac{1}{\omega(Q)}\int_Q |f(y)|\omega(y)dy.$$
 For $0<\alpha<n, r\geq 1,$ the fractional maximal operator $M_{\alpha,r}$ is defined by
 $$M_{\alpha,r}(f)(x)=\sup_{x\in Q}\left( \frac{1}{|Q|^{1-\frac{\alpha r}{n}}}\int_Q |f(y)|^rdy\right)^{\frac{1}{r}};$$
and the fractional weighted maximal operator $M_{\alpha,r,\omega}$ is defined by
  $$M_{\alpha,r,\omega}(f)(x)=\sup_{x\in Q}\left( \frac{1}{\omega(Q)^{1-\frac{\alpha r}{n}}}\int_Q |f(y)|^r\omega(y)dy\right)^{\frac{1}{r}}.$$
   \\
  \quad For any $f\in L^p(R^n), p\geq 1,$ the sharp maximal function $M^{\sharp}_Lf$ associated the generalized approximations to the identity $\{e^{-tL},t>0\}$ is given by
 \begin{equation*}
   M^{\sharp}_Lf(x)=\sup_{x\in Q} \frac{1}{|Q|}\int_Q |f(y)-e^{-t_QL}f(y)|dy
\end{equation*}
 where $t_Q=r^2_Q$ and $r_Q$ is the radius of the ball $Q$.
\par In the above definitions, the supremum is taken over all cubes $Q$ containing $x.$

\end{definition}

\begin{definition}
A weight function $\omega$ is in the Muckenhoupt class $A_p$
with $1 < p < \infty$ if for every cube $Q$ in $R^n$, there exists a positive constant
$C$ which is independent of $Q$ such that
\begin{equation*}
   \left(\frac{1}{|Q|}\int_Q\omega(x)dx\right)\left(\frac{1}{|Q|}\int_Q\omega(x)^{-\frac{1}{p-1}}dx\right)^{p-1}\leq C.
\end{equation*}
When $p=1, \omega\in A_1,$ if
\begin{equation*}
   \left(\frac{1}{|Q|}\int_Q\omega(x)dx\right)\leq C {\texttt{ess}\,\inf}_{x\in Q}\, \omega(x).
\end{equation*}

When $p=\infty, \omega\in A_\infty,$ if there exist positive constants $\delta$ and $C$ such that given a cube $Q$ and $E$ is a measurable subset of $Q$, then
\begin{equation*}
   \frac{\omega(E)}{\omega(Q)}\leq C\left(\frac{|E|}{|Q|}\right)^\delta .
\end{equation*}
\end{definition}

\begin{definition}
A weight function $\omega$ belongs to $A_{p,q}$ for $1 < p < q <\infty$ if for every cube $Q$ in $R^n$, there exists a positive constant $C$ which is
independent of $Q$ such that
\begin{equation*}
   \left(\frac{1}{|Q|}\int_Q\omega(x)^qdx\right)^{\frac{1}{q}}\left(\frac{1}{|Q|}\int_Q\omega(x)^{-p^{'}}dx\right)^{\frac{1}{p^{'}}}\leq C.
\end{equation*}
where $p^{'}$ denotes the conjugate exponent of $p>1,$ that is $1/p+1/p^{'}=1.$
\end{definition}

\begin{definition}
A weight function $\omega$ belongs to the reverse H\"{o}lder class
$RH_r$ if there exist two constants $r > 1$ and $C > 0$ such that the following
reverse H\"{o}lder inequality
\begin{equation*}
   \left(\frac{1}{|Q|}\int_Q\omega(x)^rdx\right)^{\frac{1}{r}}\leq C   \left(\frac{1}{|Q|}\int_Q\omega(x)dx\right)
\end{equation*}
holds for every cube $Q$ in $R^n$.
\end{definition}
It is well known that if $\omega\in  A_p$ with $1 \leq p < \infty$, then there exists
$r > 1$ such that $\omega\in RH_r.$ It follows from H\"{o}lder¡¯s inequality that $\omega \in  RH_r$
implies $\omega \in RH_s$ for all $1 < s < r.$ Moreover, if $\omega\in  RH_r, r > 1,$ then we
have $\omega\in RH_{r+\epsilon}$ for some $ \epsilon> 0.$ We thus write $r_w= \sup\{r > 1 : \omega \in RH_r\}$
to denote the critical index of $\omega$ for the reverse H\"{o}lder condition.

We will make use of the following lemmas. We first provide a weighted version of the local good $\lambda$ inequality for $ M^{\sharp}_L$ which allow us to obtain an analog of the classical Fefferman-Stein(see \cite{fefferman,coifman}) estimate on weighted Morrey spaces.
\begin{lemma}(\cite{martell})\label{1}
Assume that the semigroup $e^{-tL}$ has  a kernel $p_t(x,y)$ which satisfies the upper bound  \eqref{1.1}. Take $\lambda>0, f\in L^1_0(R^n)$ and a ball $Q_0$ such that there exists $x_0\in Q_0$ with $M f(x_0)\leq \lambda.$ Then, for every $\omega \in A_\infty,0<\eta<1,$ we can find $\gamma>0$(independent of $\lambda,Q_0, f,x_0$) and constant $C_{\omega,}r>0$(which only depend on $\omega$) .
 $$\omega\{x\in Q_0: M f>A \lambda, M^{\sharp}_Lf(x)\leq \gamma\lambda\}\leq C_{\omega}\eta^r \omega(Q_0).$$
 where $A>1$ is a fixed constant which depends only on $n.$
\end{lemma}

 As a consequence, by using the standard arguments, we have the following estimates:
 \par For every $f\in L^{p,k}(\mu,\nu)$, with $1<p<\infty.$
 If $\mu,\nu\in A_\infty, 1<p<\infty, 0\leq k<1.$
 $$||f||_{L^{p,k}(\mu,\nu)}\leq ||Mf||_{L^{p,k}(\mu,\nu)}\leq C||M^{\sharp}_Lf||_{L^{p,k}(\mu,\nu)}$$
 In particular, when $\mu=\nu=\omega$ and $\omega \in A_\infty,$ we have
  $$||f||_{L^{p,k}(\omega)}\leq ||Mf||_{L^{p,k}(\omega)}\leq C||M^{\sharp}_Lf||_{L^{p,k}(\omega)}$$

\begin{lemma}(\cite{wang})\label{2}
Let $0<\alpha <n, 1<p<\frac{n}{\alpha},\frac{1}{q}=\frac{1}{p}-\frac{\alpha}{n}$ and $\omega^{q/p}\in A_1.$
 Then if $0<k<p/q$ and $r_\omega> \frac{1-k}{p/q-k}$, we have
$$||M_{\alpha, 1}f||_{L^{q,kq/p}(\omega^{q/p},\omega)}\leq C ||f||_{L^{p,k}(\omega)}.$$
The same conclusion still hold for $I_\alpha$.
\end{lemma}

\begin{lemma}(\cite{wang})\label{21}
Let $0<\alpha <n, 1<p<\frac{n}{\alpha},\frac{1}{q}=\frac{1}{p}-\frac{\alpha}{n}$ and $\omega^{q/p}\in A_1.$
 Then if $0<k<p/q, 1<r<p$ and $r_\omega> \frac{1-k}{p/q-k}$, we have
$$||M_{r,\omega}f||_{L^{q,kq/p}(\omega^{q/p},\omega)}\leq C ||f||_{L^{q,kq/p}(\omega^{q/p},\omega)}.$$
\end{lemma}

\begin{lemma}(\cite{wang})\label{3}
$0<\alpha <n, 1<p<\frac{n}{\alpha},\frac{1}{q}=\frac{1}{p}-\frac{\alpha}{n}$, $0<k<p/q$, $\omega\in A_\infty.$
For any $1<r<p,$ we have
$$||M_{\alpha, r,\omega}f||_{L^{q,kq/p}(\omega)}\leq C ||f||_{L^{p,k}(\omega)}.$$
\end{lemma}
\remark{By checking the proof of Lemma \ref{2}, Lemma \ref{21}, Lemma \ref{3}, we know the three lemmas above still hold when $k=0$.}


\begin{lemma}\label{5}
Let $0<\alpha <n, 1<p<\frac{n}{\alpha},\frac{1}{q}=\frac{1}{p}-\frac{\alpha}{n}$ and $\omega^{q/p}\in A_1.$
 Then if $0\leq k<p/q$ and $r_\omega> \frac{1-k}{p/q-k}$, we have
$$||L^{-\alpha/2}f||_{L^{q,kq/p}(\omega^{q/p},\omega)}\leq C ||f||_{L^{p,k}(\omega)}.$$
\end{lemma}
\begin{proof}
Since the semigroup $e^{-tL}$ has a kernel $p_t(x,y)$ which satisfies the upper bound \eqref{1.1},
it is easy to check that  $L^{-\alpha/2}(f)(x)\leq C I_\alpha(|f|)(x) $ for all $x\in R^n.$ Using the boundedness property of $I_\alpha$
on weighted Morrey space(see Lemma \ref{2}), we have
$$||L^{-\alpha/2}f||_{L^{q,kq/p}(\omega^{q/p},\omega)}\leq ||I_\alpha f||_{L^{q,kq/p}(\omega^{q/p},\omega)}  \leq C||f||_{L^{p,k}(\omega)},$$
where $1<p<\frac{n}{\alpha}$ and $\frac{1}{q}=\frac{1}{p}-\frac{\alpha}{n}.$

\end{proof}
\remark{Since  $I_\alpha$ is  weak type  $(1, n/(n-\alpha)).$ From the proof  Lemma \ref{5}, we can get
$L^{-\alpha/2}$ is also  weak  type $(1, n/(n-\alpha)).$ }
\begin{lemma}(\cite{ddsy})\label{6}
Assume that the semigroup $e^{-tL}$ has  a kernel $p_t(x,y)$ which satisfies the upper bound \eqref{1.1}. Then for $0<\alpha <n,$ the difference operator $L^{-\frac{\alpha}{2}}-e^{-tL}L^{-\frac{\alpha}{2}}$ has an associated kernel $K_{\alpha,t}(x,y)$ which satisfies
$$K_{\alpha,t}(x,y)\leq \frac{C}{|x-y|^{n-\alpha}}\frac{t}{|x-y|^2}$$
for some positive constant $C.$
\end{lemma}
\begin{lemma}\label{7}
Assume that the semigroup $e^{-tL}$ has a kernel $p_t(x,y)$ which satisfies the upper bound \eqref{1.1},
 and let $b\in BMO(\omega), \omega\in A_1.$ Then, for every function $f\in L^p(R^n), p>1 $ and for all $ x\in R^n, $ we have
 $$\sup_{x\in Q} \frac{1}{|Q|}\int_Q|e^{-t_QL}(b(y)-b_Q)f(y)|dy\leq C ||b||_{BMO(\omega)}\omega(x)M_{r,\omega}(f)(x).$$
where $M_{r,\omega}(f)(x)=M_\omega(|f|^r)^{\frac 1r}(x)$ with $1<r<\infty.$
\end{lemma}
\begin{proof}
For any $f\in L^p(R^n), 1<p<\infty$ and $x\in Q.$  We have
\begin{equation*}
\begin{split}
&\frac{1}{|Q|}\int_Q |e^{-t_Q L}((b(\cdot)-b_Q)f)(y)|dy\\
&\leq \frac{1}{|Q|}\int_Q \int_{R^n} |p_{t_Q}(y,z)||(b(z)-b_Q)f(z)|dzdy\\
&\leq \frac{1}{|Q|}\int_Q \int_{2Q} |p_{t_Q}(y,z)||(b(z)-b_Q)f(z)|dzdy\\
&\quad + \frac{1}{|Q|}\int_Q \sum_{k=1}^\infty\int_{2^{k+1}Q\setminus 2^{k}Q} |p_{t_Q}(y,z)||(b(z)-b_Q)f(z)|dzdy\\
&\doteq \mathcal{M}+ \mathcal{N}.
\end{split}
\end{equation*}
Since for any $y\in Q$ and $z\in 2Q.$ We have
$$ |p_{t_Q}(y,z)|\leq C t^{-\frac{n}{2}}_{Q}\leq C \frac{1}{|2Q|}.$$
Thus,
\begin{equation*}
\begin{split}
\mathcal{M}&\leq C\frac{1}{|2Q|} \int_{2Q}|(b(z)-b_Q)f(z)|dz\\
&\leq C \frac{1}{|2Q|}\left( \int_{2Q}||b(z)-b_Q|^{r^{'}}\omega(z)^{1-r^{'}} dz\right)^{\frac{1}{r^{'}}} \left( \int_{2Q}|f(z)|^r\omega(z)dz\right)^{\frac{1}{r}}\\
&\leq C ||b||_{BMO(\omega)}\frac{\omega(2Q)}{|2Q|} \left( \frac{1}{\omega(2Q)}\int_{2Q}|f(z)|^r\omega(z)dz\right)^{\frac{1}{r}}\\
&\leq C ||b||_{BMO(\omega)}\omega(x) M_{r,\omega}f(x).
\end{split}
\end{equation*}
Moreover, for any $y\in Q$ and $z\in 2^{k+1}Q\setminus 2^{k}Q$, we have $|y-z|\geq 2^{k-1}r_Q$  and $|p_{t_Q}|\leq C \frac{e^{-c2^{2(k-1)}}2^{(k+1)n}}{|2^{k+1}Q|}$.
\begin{equation*}
\begin{split}
\mathcal{N}&=\frac{1}{|Q|}\int_Q \sum_{k=1}^\infty\int_{2^{k+1}Q\setminus 2^{k}Q} |p_{t_Q}(y,z)||(b(z)-b_Q)f(z)|dzdy\\
&\leq C \sum_{k=1}^\infty \frac{e^{-c2^{2(k-1)}}2^{(k+1)n}}{|2^{k+1}Q|}\int_{2^{k+1}Q} |(b(z)-b_Q)f(z)|dz\\
&\leq C \sum_{k=1}^\infty \frac{e^{-c2^{2(k-1)}}2^{(k+1)n}}{|2^{k+1}Q|}\int_{2^{k+1}Q} |(b(z)-b_{2^{k+1}Q})f(z)|dz\\
&\quad +C \sum_{k=1}^\infty \frac{e^{-c2^{2(k-1)}}2^{(k+1)n}}{|2^{k+1}Q|}\int_{2^{k+1}Q} |(b_{2^{k+1}Q}-b_{2Q})f(z)|dz\\
&\doteq \mathcal{N}_1+\mathcal{N}_2.
\end{split}
\end{equation*}
We  estimate each term in turn. For $\mathcal{N}_1,$ we apply H\"{o}lder's inequalities with exponent $r$. Then we have

\begin{equation*}
\begin{split}
\mathcal{N}_1&\leq C \sum_{k=1}^\infty \frac{e^{-c2^{2(k-1)}}2^{(k+1)n}}{|2^{k+1}Q|}\left( \int_{2^{k+1}Q}||b(z)-b_Q|^{r^{'}}\omega(z)^{1-r^{'}} dz\right)^{\frac{1}{r^{'}}} \left( \int_{2^{k+1}Q}|f(z)|^r\omega(z)dz\right)^{\frac{1}{r}}\\
&\leq C \sum_{k=1}^\infty 2^{(k+1)n}e^{-c2^{2(k-1)}}||b||_{BMO(\omega)}\frac{\omega(2^{k+1}Q)}{|2^{k+1}Q|} \left( \frac{1}{\omega(2^{k+1}Q)}\int_{2^{k+1}Q}|f(z)|^r\omega(z)dz\right)^{\frac{1}{r}}\\
&\leq C ||b||_{BMO(\omega)}\omega(x) M_{r,\omega}f(x).
\end{split}
\end{equation*}

Since $\omega\in A_1,$ then $|b_{2^{k+1Q}}-b_{2Q}|\leq C k \,\omega (x)||b||_{BMO(\omega)}.$  This fact together with the H\"{o}lder inequality implies
\begin{equation*}
\begin{split}
\mathcal{N}_2&\leq C \sum_{k=1}^\infty 2^{(k+1)n}e^{-c2^{2(k-1)}}\frac{k}{|2^{k+1}Q|}\omega(x)||b||_{BMO(\omega)}\int_Q|f(z)|dz\\
&\leq C \sum_{k=1}^\infty k2^{(k+1)n}e^{-c2^{2(k-1)}}\omega(x)||b||_{BMO(\omega)}\left(\frac{1}{|2^{k+1}Q|}\int_{2^{k+1}Q} |f(z)|^rdz\right)^{\frac{1}{r}}\\
&= C \sum_{k=1}^\infty k2^{(k+1)n}e^{-c2^{2(k-1)}}\omega(x)||b||_{BMO(\omega)} \left(\frac{\omega(2^{k+1}Q)}{|2^{k+1}Q|}\frac{1}{\omega(2^{k+1}Q)}\int_{2^{k+1}Q} |f(z)|^rdz\right)^{\frac{1}{r}}\\
&\leq C \sum_{k=1}^\infty k2^{(k+1)n}e^{-c2^{2(k-1)}}\omega(x)||b||_{BMO(\omega)}\left(\frac{1}{\omega(2^{k+1}Q)}\int_{2^{k+1}Q} |f(z)|^r \omega(x)dz\right)^{\frac{1}{r}}\\
&\leq C ||b||_{BMO(\omega)}\omega(x) M_{r,\omega}f(x).
\end{split}
\end{equation*}

Then Lemma \ref{7} is proved.

\end{proof}

\begin{lemma}\label{8}
Let $0<\alpha <n,$ $\omega \in A_1$ and $b\in BMO(\omega).$ Then
for all $r>1$ and for all $x\in R^n,$ we have
\begin{equation}\label{2.1}
 \begin{split}
   &M^{\sharp}_L([b, L^{-\frac{\alpha}{2}}]f)(x)\\
  &\leq  C ||b||_{BMO(\omega)}\left(\omega(x)M_{r,\omega}(L^{-\frac{\alpha}{2}}f)(x)+\omega(x)^{1-\frac{\alpha}{n}}M_{\alpha, r,\omega}(f)(x)+\omega(x)M_{\alpha, 1}(f)(x)\right).
\end{split}
\end{equation}

\end{lemma}

\begin{proof}
For any given $x\in R^n,$ fix a ball $Q=Q(x_0, r_B)$ which contain $x.$ we decompose $f=f_1+f_2, $ where $f_1=f\chi_{2Q}.$ Observe that
$$[b, L^{-\frac{\alpha}{2}}]f(x)=(b-b_Q)L^{-\frac{\alpha}{2}}f-L^{-\frac{\alpha}{2}}(b-b_Q)f_1-L^{-\frac{\alpha}{2}}(b-b_Q)f_2$$
and
$$e^{-t_Q L}([b,L^{-\frac{\alpha}{2}}]f)=e^{-t_Q L}[(b-b_Q)L^{-\frac{\alpha}{2}}f-L^{-\frac{\alpha}{2}}(b-b_Q)f_1-L^{-\frac{\alpha}{2}}(b-b_Q)f_2].$$
Then
\begin{equation*}
\begin{split}
&\frac{1}{|Q|}\int_Q |[b,L^{-\frac{\alpha}{2}}]f(y)-e^{-t_Q L}[b,L^{-\frac{\alpha}{2}}]f(y)|dy\\
&\leq \frac{1}{|Q|}\int_Q |(b(y)-b_Q)L^{-\frac{\alpha}{2}}f(y)|dy+ \frac{1}{|Q|}\int_Q |L^{-\frac{\alpha}{2}}(b(y)-b_Q)f_1)(y)|dy\\
& \quad +\frac{1}{|Q|}\int_Q |e^{-t_QL}((b(y)-b_Q)L^{-\frac{\alpha}{2}}f)(y)|dy+ \frac{1}{|Q|}\int_Q |e^{-t_QL}L^{-\frac{\alpha}{2}}((b(y)-b_Q)f_1(y))|dy\\
& \quad +\frac{1}{|Q|}\int_Q |(L^{-\frac{\alpha}{2}}-e^{-t_QL}L^{-\frac{\alpha}{2}})((b-b_Q)f_2)(y)|dy\\
&\doteq I +II+ III +IV+V.
\end{split}
\end{equation*}
We estimate each term separately.
\par Since $\omega\in A_1,$ then it follows from H\"{o}lder's inequality
\begin{equation*}
\begin{split}
I&\leq \frac{1}{|Q|}\int_Q |(b(y)-b_Q)L^{-\frac{\alpha}{2}}f(y)|dy\\
 &\leq\frac{1}{|Q|}\left(\int_Q |(b(y)-b_Q)|^{r^{'}}\omega^{1-r^{'}}dy\right)^{\frac{1}{r^{'}}}\left(\int_Q |L^{-\frac{\alpha}{2}}f(y)|^{r}\omega(y) dy\right)^{\frac{1}{r}}\\
&\leq C||b||_{BMO(\omega)}\frac{\omega(Q)}{|Q|}\left(\frac{1}{\omega(Q)}\int_Q |L^{-\frac{\alpha}{2}}f(y)|^{r}\omega(y) dy\right)^{\frac{1}{r}}\\
& \leq C||b||_{BMO(\omega)}\omega(x)M_{r,\omega}(L^{-\frac{\alpha}{2}}f)(x).
\end{split}
\end{equation*}

Applying Kolmogorov's inequality(see\cite{francia}, page 485), H\"{o}lder's inequality and  the continuity  of $L^{-\alpha/2}$, we thus have
\begin{equation*}
\begin{split}
II &=  \frac{1}{|Q|}\int_Q |L^{-\frac{\alpha}{2}}(b(y)-b_Q)f_1)(y)|dy\\
& \leq C\frac{1}{|Q|^{1-\frac{\alpha}{n}}}||L^{-\frac{\alpha}{2}}(b(y)-b_{2Q})f_1||_{L^{\frac{n}{n-\alpha},\infty}}\\
&\leq C  \frac{1}{|Q|^{1-\frac{\alpha}{n}}}\int_Q (b(y)-b_{2Q})f_1(y)dy\\
&\leq C  \frac{1}{|Q|^{1-\frac{\alpha}{n}}} \left(\int_Q |b(y)-b_{2Q}|^{r^{'}}\omega^{1-r^{'}} dy \right)^{\frac{1}{r^{'}}} \left( \int_Q |f(y)|^r\omega(y)dy\right)^{\frac{1}{r}}\\
& \leq C||b||_{BMO(\omega)}\frac{w(2Q)^{1-\frac{\alpha}{n}}}{|2Q|^{1-\frac{\alpha}{n}}} \left( \frac{1}{w(2Q)^{1-\frac{r\alpha}{n}}} \int_Q |f(y)|^r\omega(y)dy\right)^{\frac{1}{r}}\\
& \leq C||b||_{BMO(\omega)}\omega(x)^{1-\frac{\alpha}{n}}M_{\alpha,r,\omega}(f)(x).
\end{split}
\end{equation*}
By Lemma \ref{7}, we have
$$III \leq C ||b||_{BMO(\omega)}\omega(x)M_{r,\omega}(L^{-\frac{\alpha}{2}}f)(x).$$
For IV, use the estimate obtained in $II$, we get

\begin{equation*}
\begin{split}
IV&\leq  \frac{1}{|Q|}\int_Q \int_{2Q}|p_{tQ}(y,z)||b(z)-b_Q||f(z)|dzdy\\
&\leq  \frac{1}{|2Q|} \int_{2Q}L ^{-\frac{\alpha}{2}}((b(z)-b_Q))f(z)|dz\\
& \leq C||b||_{BMO(\omega)}\omega(x)^{1-\frac{\alpha}{n}}M_{\alpha,r,\omega}(f)(x).
\end{split}
\end{equation*}

By virtue of Lemma \ref{6}, we have

\begin{equation*}
\begin{split}
V&\leq \frac{1}{|Q|}\int_Q \int_{(2Q)^c}|K_{\alpha,tQ}(y,z)||(b(z)-b_Q)f(z)|dzdy\\
&\leq C \sum_{k=1}^\infty \int_{2^kr_Q\leq |x_0-z|<2^{k+1}r_Q}\frac{1}{|x_0-z|^{n-\alpha}}\frac{r_Q}{|x_0-z|}|(b(z)-b_Q)f(z)|dz\\
&\leq C \sum_{k=1}^\infty 2^{-k} \frac{1}{|2^{k+1}Q|^{1-\frac{\alpha}{n}}}\int_{2^{k+1}Q} |(b(z)-b_Q)f(z)|dz\\
 &\leq C \sum_{k=1}^\infty 2^{-k} \frac{1}{|2^{k+1}Q|^{1-\frac{\alpha}{n}}}\int_{2^{k+1}Q} |(b(z)-b_{2^{k+1}Q})f(z)|dz\\
 &\quad + C \sum_{k=1}^\infty 2^{-k}(b_{2^{k+1}Q}-b_Q) \frac{1}{|2^{k+1}Q|^{1-\frac{\alpha}{n}}}\int_{2^{k+1}Q} |f(z)|dz\\
 &\doteq VI +VII.
\end{split}
\end{equation*}

For VI, apply the same arguments as in $II$, we get
\begin{equation*}
\begin{split}
VI&\leq C ||b||_{BMO(\omega)} \sum_{k=1}^\infty 2^{-k}\omega(x)^{1-\frac{\alpha}{n}}M_{\alpha,r,\omega}(f)(x)\\
&\leq C ||b||_{BMO(\omega)} \omega(x)^{1-\frac{\alpha}{n}}M_{\alpha,r,\omega}(f)(x).
\end{split}
\end{equation*}

Since $\omega \in A_1,$ then
$|b_{2^{k+1Q}}-b_{2Q}|\leq C k \,\omega (x)||b||_{BMO(\omega)}.$
Thus,
\begin{equation*}
\begin{split}
VII&\leq  C ||b||_{BMO(\omega)} \sum_{k=1}^\infty 2^{-k}k\, \omega(x)M_{\alpha,1}(f)(x)\\
&\leq  C ||b||_{BMO(\omega)} \omega(x) M_{\alpha,1}(f)(x).
\end{split}
\end{equation*}
Then
\begin{equation*}
\begin{split}
V&\leq C ||b||_{BMO(\omega)}\left( \omega(x)^{1-\frac{\alpha}{n}}M_{\alpha,r,\omega}(f)(x)+\omega(x) M_{\alpha,1}(f)(x)\right).
\end{split}
\end{equation*}
Combining the above estimates I, II, III, IV, and V, we get \eqref{2.1}. The proof of Lemma \ref{8} is complete.
\end{proof}

\section{PROOFS OF THE MAIN RESULTS}
In this section we prove our main results.
We start with the proof of Theorem \ref{thm1.1}.

\begin{proof}
 $(a)\Rightarrow (b):$  Applying Lemma \ref{1} and Lemma \ref{8}, we get
\begin{equation*}
\begin{split}
&||[b, L^{-\frac{\alpha}{2}}]f||_{L^{q,kq/p}(\omega^{1-(1-\alpha/n)q},\omega)}\\
&\leq ||M^{\sharp}_L([b, L^{-\frac{\alpha}{2}}]f)||_{L^{q,kq/p}(\omega^{1-(1-\alpha/n)q},\omega)}\\
&\leq C  ||b||_{BMO(\omega)}\big( || \omega(x)M_{r,\omega}(L^{-\frac{\alpha}{2}}f)||_{L^{q,kq/p}(\omega^{1-(1-\alpha/n)q},\omega)}\\
&\quad + ||  \omega^{1-\frac{\alpha}{n}}M_{\alpha,r,\omega}f||_{L^{q,kq/p}(\omega^{1-(1-\alpha/n)q},\omega)}\\
&\quad +|| \omega M_{\alpha,1}f||_{L^{q,kq/p}(\omega^{1-(1-\alpha/n)q},\omega)}\big)\\
&\leq C  ||b||_{BMO(\omega)}\big(|| M_{r,\omega}(L^{-\frac{\alpha}{2}}f)||_{L^{q,kq/p}(\omega^{q/p},\omega)} +||  M_{\alpha,r,\omega}f||_{L^{q,kq/p}(\omega)}\\
&\quad +||  M_{\alpha,1}f||_{L^{q,kq/p}(\omega^{q/p},\omega)}\big).\\
\end{split}
\end{equation*}

Since $0\leq k<p/q, \omega^{q/p}\in A_1$ and $r_\omega > \frac{1-k}{p/q-k},$ by making use of Lemma \ref{2}, Lemma \ref{21} and Lemma \ref{3}, then we obtain

\begin{equation*}
\begin{split}
&||[b, L^{-\frac{\alpha}{2}}]f||_{L^{q,kq/p}(\omega^{1-(1-\alpha/n)q},\omega)}\\
&\leq C  ||b||_{BMO(\omega)}\big(||L^{-\frac{\alpha}{2}}f||_{L^{q,kq/p}(\omega^{q/p},\omega)}+||f||_{L^{p,k}(\omega)}\big)\\
&\leq C  ||b||_{BMO(\omega)}||f||_{L^{p,k}(\omega)}
\end{split}
\end{equation*}
The last inequality follows from Lemma \ref{5}.
This completes the proof of  $(a)\Rightarrow (b).$
\par
$(b)\Rightarrow (a):$
  Let $L = -\Delta$ be the Laplacian on $R^n$,
then $L^{-{\alpha/ 2}}$ is the classical fractional
integral $I_\alpha$. Choose $Z_0 \in R^n$ so that $|Z_0|=3.$ For $x\in Q(Z_0,2),$ $|x|^{-\alpha+n}$ can be written as the absolutely convergent Fourier series,
$|x|^{-\alpha+n}=\sum_{m\in Z_n}a_m e^{i<\nu_m,x>}$ with $\sum_m |a_m|<\infty$ since $|x|^{-\alpha+n} \in C^\infty(Q(Z_0,2))$. For any $x_0\in R^n$
and $\rho>0,$ let $Q=Q(x_0,\rho)$ and $Q_{Z_0}=Q(x_0+Z_0 \rho,\rho),$
\begin{equation*}
\begin{split}
\int_Q |b(x)-b_{Q_{Z_0}}|dx&=\frac{1}{|Q_{Z_0}|}\int_Q |\int_{Q_{Z_0}}(b(x)-b(y))dy|dx\\
&=\frac{1}{\rho^n}\int_Q s(x)\left(\int_{Q_{Z_0}}(b(x)-b(y))|x-y|^{-\alpha+n}|x-y|^{n-\alpha}dy \right)dx,
\end{split}
\end{equation*}
where $s(x)=\overline{\texttt{sgn}(\int_{Q_{Z_0}}(b(x)-b(y))dy)} .$ Fix $x\in Q$ and $y\in Q_{Z_0}$ we have
$\frac{y-x}{\rho}\in Q(Z_0,2)$, hence, we have
\begin{equation*}
\begin{split}
&\frac{\rho^{-\alpha+n}}{\rho^n}\int_Q s(x)\left(\int_{Q_{Z_0}}(b(x)-b(y))|x-y|^{-\alpha+n}(\frac{|x-y|}{\rho})^{n-\alpha}dy \right)dx\\
&= \rho^{-\alpha}\sum_{m\in Z^n}a_m \int_Q s(x)\left(\int_{Q_{Z_0}}(b(x)-b(y))|x-y|^{-\alpha+n}e^{i<\nu_m,y/\rho>}dy \right)e^{-i<\nu_m,x/\rho>}dx\\
&\leq \rho^{-\alpha} \left|\sum_{m\in Z^n}|a_m|  \int_Q s(x)[b,L^{-{\alpha/ 2}}](\chi_{Q_{Z_0}}e^{i<\nu_m,\cdot/\rho>})\chi_Q(x)e^{-i<\nu_m,x/\rho>}dx\right| \\
&\leq \rho^{-\alpha} \sum_{m\in Z^n}|a_m| ||[b,L^{-{\alpha/ 2}}](\chi_{Q_{Z_0}}e^{i<\nu_m,\cdot/\rho>})||_{L^{q,0}(\omega^{1-(1-\alpha/ n)q},\omega)} \left(\int_Q \omega(x)^{q^{'}[(1-\alpha/ n)-1/q]}dx\right)^{\frac{1}{q^{'}}}\\
&\leq C\rho^{-\alpha} \sum_{m\in Z^n}|a_m| ||\chi_{Q_{Z_0}}||_{L^{p,0}(\omega)}\left(\int_Q \omega(x)^{q^{'}(1/q^{'}-\alpha/ n)}dx\right)^{\frac{1}{q^{'}}}\\
&\leq C \omega(Q)^{1/p+1/q^{'}-\alpha/ n}\\
&\leq C \omega(Q).
\end{split}
\end{equation*}

This implies $b\in BMO(\omega).$ Thus Theorem \ref{thm1.1} is proved.
\end{proof}

Similarly, to prove Theorem \ref{thm1.2} we  need the following lemmas.


\begin{lemma}\label{10}
Let $0<\alpha <n,1<r,s<\infty$ such that $rs<p<n/\alpha$ and $b\in BMO.$ Then
for all $r>1$ and for all $x\in R^n,$ we have
$$M^{\sharp}_L([b,L^{-{\alpha/ 2}}]f)(x)\leq C ||b||_{BMO}\left(M_{r}(L^{-\frac{\alpha}{2}}f)(x)+M_{\alpha, rs}(f)(x)\right).$$
 where $M_{r}(f)(x)=M(|f|^r)^{\frac 1r}(x)$.
\end{lemma}
\begin{proof}
The case $0<\alpha<1$ was proved by Duong and Yan (see \cite{duongyan} for details). The general case $0<\alpha<n$ follows by repeating the same steps as in Lemma \ref{8}. Since the main steps and the ideas are almost the same, here we omit the proof.
\end{proof}
\begin{lemma}(\cite{ks})\label{11}
If $0<\alpha<n, 1<p<n/\alpha, 1/q=1/p-\alpha/n,0<k<p/q$ and $\omega\in A_{p,q}$, then the fractional maximal operator
$M_{\alpha,1}$ is bounded from $L^{p,k}(\omega^p,\omega^q)$ to $L^{q,kq/p}(\omega^q)$.
\end{lemma}

\begin{lemma}(\cite{ks})\label{12}
If $0<\alpha<n, 1<p<n/\alpha, 1/q=1/p-\alpha/n,0<k<p/q$ and $\omega\in A_{p,q}$, then the fractional maximal operator
$I_\alpha$ is bounded from $L^{p,k}(\omega^p,\omega^q)$ to $L^{q,kq/p}(\omega^q)$.
\end{lemma}

\begin{lemma}(\cite{ks})\label{13}
If $1<p<\infty, 0<k<1$ and $\omega\in A_p,$ then $M$ is bounded on $L^{p,k}(\omega).$

\end{lemma}
\remark{By applying the same argument as in  Lemma \ref{5}, we know the conclusion in Lemma \ref{12} still hold for $L^{-{\alpha/ 2}}$. We omit the proof here. }
\remark{By checking the proof of Lemma \ref{11}, Lemma \ref{12}, Lemma \ref{13}, we know the three lemmas above still hold when $k=0$.}

\par Now we prove Theorem \ref{thm1.2}.
\begin{proof}
$(a)\Rightarrow (b):$  Since $ \omega^{rs}\in A_{p/rs,q/rs},$ then we get $ \omega^{q}\in A_{q/rs}$ and  $ \omega^{p}\in A_{p/rs}$.  Applying Lemma \ref{1},  Lemma \ref{10}, Lemma \ref{11}, Lemma \ref{12} and Lemma \ref{13}, we get
\begin{equation*}
\begin{split}
&||[b, L^{-{\alpha/ 2}}]f||_{L^{q,kq/p}(\omega^q)}\\
&\leq ||M^{\sharp}_L([b, L^{-\frac{\alpha}{2}}]f)||_{L^{q,kq/p}(\omega^q)}\\
&\leq C  ||b||_{BMO}\big( || M_{r}(L^{-\frac{\alpha}{2}}f)||_{L^{q,kq/p}(\omega^q)}+ ||  M_{\alpha, rs}(f)||_{L^{q,kq/p}(\omega^q)}\big)\\
&\leq C  ||b||_{BMO}\big( || L^{-\frac{\alpha}{2}}f||_{L^{q,kq/p}(\omega^q)}+||f||_{L^{p,k}(\omega^p,\omega^q)}\big)\\
&\leq C  ||b||_{BMO}||f||_{L^{p,k}(\omega^p,\omega^q)}.
\end{split}
\end{equation*}
In the last inequality, we used the fact $L^{-{\alpha/ 2}}$ is bounded from $L^{p,k}(\omega^p,\omega^q)$ to $L^{q,kq/p}(\omega^q)$(see remark 3.1).
\par
$(b)\Rightarrow (a):$    Let $L = -\Delta$ be the Laplacian on $R^n$,
then $L^{-{\alpha/ 2}}$ is the classical fractional
integral $I_\alpha$ and Let $k=0$ and weight $\omega\equiv 1,$  then $L^{p,k}(\omega^p,\omega^q)=L^p$ and $L^{q,kq/p}(\omega^q)=L^q.$
From \cite{ludingyan} we know the $(L^p, L^q)$ bounedness of $[b,I_\alpha]$ implies $b\in BMO$. Thus Theorem \ref{thm1.2} is proved.
\end{proof}

\end{document}